\documentclass{amsproc}
\usepackage[utf8]{inputenc}
  \usepackage{amsmath}
\usepackage{enumerate}
  \usepackage{amssymb}
  \usepackage{amsthm}
  \usepackage{bm}
  \usepackage{graphicx}
  \usepackage{color}
\usepackage{accents}
\usepackage{epsfig}
\usepackage{amsmath,amssymb}
\hoffset - 0.5 in
\newtheorem{Th}{Theorem}[section]
\newtheorem{lem}[Th]{Lemma}
\newtheorem{Cor}[Th]{Corollary}
\newtheorem{Prop}[Th]{Proposition}

\theoremstyle{definition}
\newtheorem{Def}[Th]{Definition}

\theoremstyle{remark}
\newtheorem*{rem}{\bf Remarks}

\numberwithin{equation}{section}

\newcommand{\tend}[3][]{\xrightarrow[#2\to#3]{#1}}

\newcommand{\egdef}{\stackrel{\textrm {def}}{=}}
\newcommand{\ds}{\displaystyle}

\newcommand{\printdate}{\today}

\title{Some Notes on Flat Polynomials  }

\author{e. H. el Abdalaoui}
\address{Normandie University, University of Rouen
  Department of Mathematics, LMRS  UMR 60 85 CNRS\\
Avenue de l'Universit\'e, BP.12
76801 Saint Etienne du Rouvray - France .}
\email{elhoucein.elabdalaoui@univ-rouen.fr}
\urladdr{http://www.univ-rouen.fr/LMRS/Persopage/Elabdalaoui/}
\author{M. G. Nadkarni}
\address{Department of Mathematics, University of Mumbai, Vidyanagari, Kalina,  Mumbai, 400098, India}
\email{mgnadkarni@gmail.com}
\urladdr{http://insaindia.org/detail.php?id=N91-1080}


\dedicatory{}

\begin{document}
\maketitle
\begin{abstract} Connection of flat polynomials with some spectral questions in ergodic theory is discussed. A necessary condition for a sequence of polynomials of the type $\frac{1}{\sqrt{N}} \big(1 +\sum_{j=1}^{N-1} z^{n_j}\big)$
to be flat in almost everywhere sense is given, which contrasts with a similar necessary condition for a sequence of polynomials to be ultraflat.
\end{abstract}

\section{Introduction}\label{intro} A sequence $P_j, j =1,2,\cdots$ of trigonometric polynomials of $L^2$ norm one is said to be flat if the sequence $|P_j|,$  $j=1,2,\cdots$ of their absolute values
converges to the constant function 1 in some sense. The sense of convergence varies according to the situation. Littlewood problem requires that the convergence be in the sup norm and the individual polynomials in the sequence have coefficients of same absolute value \cite{Littlewood},\cite{Kahane}, \cite{Queffelec}, \cite{Bom},\cite{Newman}, \cite{Borwein1},\cite{Borwein2},
\cite{Borwein3},\cite{Aistleitner}. When the convergence  required is uniform, the sequence of polynomial is often called ultraflat.  In problems connected with Barker sequences, the $L^4$ norm of the polynomials is required to be close to 1 \cite{dow}. Our interest in flat sequence of polynomials comes from spectral questions about rank one transformations in ergodic theory where the polynomials are required to have nonnegative coefficients and
their $L^1$ norms close to one or they converge in absolute value to 1 almost everywhere. It is an open question if such a flat sequence of polynomials exists in a non-trivial sense.  \cite{Bourgain},\cite{Guenais}, \cite{Abd-Nad}, \cite{Host-Mela-Parreau}, \cite{Nad}. In this note we give a necessary condition for a sequence of  absolute values of such polynomials to converge almost everywhere to 1.

\section{Ultraflat Sequence of Polynomials}

\begin{Def}\label{def1}
 Let $S^1$ denote the circle group and let $dz$ denote the normalized Lebesgue measure on it. A sequence $P_n, n=1,2,\cdots$  of analytic trigonometric polynomials with $L^2(S^1,dz)$ norm 1 and their constant terms positive, is said to be ultraflat if $| P_j(z)| \rightarrow 1$ uniformly as $j\rightarrow \infty$. It is said to be flat a.e. $(dz)$ if
 $|P_n(z)|$ converges to 1 a.e. $(dz)$. \\
\end{Def}

  The sequence $P_j(z)=1, j=1,2,\cdot$ is obviously ultraflat. More generally let
   $$P_j(z)= 1 + Z_j(z), j =1,2,\cdots$$
  where each $Z_j$ is an analytic trigonometric polynomial with zero constant term and such that $\mid\mid  Z_j\mid\mid_{\infty} \rightarrow 0$ as $j \rightarrow \infty$. Then
  $\frac{P_j(z)}{\mid\mid P_j\mid\mid_2}, j =1,2,\cdots$ is a sequence of ultraflat polynomials which we call a perturbation of the sequence of constant ultraflat polynomials
   $P_j, j =1,2,\cdots$.\\

  Let $P(z)$ be a polynomial and let $E$ denote its set of its zeros strictly inside the unit disk, $F$ the set of zeros of $P$ on or outside unit circle.  Let
  $$B(z) = \gamma\prod_{\alpha\in E}\frac{z -\alpha}{1-{\overline {\alpha}}z},  Q(z) = {\overline{\gamma}}\prod_{\alpha\in E}(1 - {\overline{\alpha}}z)\prod_{\alpha \in F}(z -\alpha),$$
   where $\gamma$ is a constant of absolute value 1 such that constant term of $Q(z)$ is positive. The function $B(z)$ which is of absolute value 1 on $S^1$ is called the inner part of $P(z)$ and $Q(z)$ the outer part of $P(z)$. We note that $P = BQ$. This factoring of $P$ is in fact  Beurling's factoring of an $H^2$ function applied to the polynomial $P$. A function of the form $B$ is called finite Blaschke product.\\

  \begin{Prop}\label{prop1}
    Given any sequence $P_j, j=1,2,\cdots$ of ultraflat polynomials, their outer parts $Q_j, j=1,2,\cdots$ form a sequence of ultraflat polynomials which is a perturbation of the constant ultraflat sequence.
  Moreover for all $j$, $|P_j(z)| = |Q_j(z)|$ on $S^1$.\\
\end{Prop}

\begin{proof}
   That $\mid P_j(z)\mid = \mid Q_j(z)\mid, j =1,2 \cdots$ follows from the construction of inner and outer factors of $P_j$. Since $\mid P_j\mid, j =1,2,\cdots$ converges to 1 uniformly, we may assume without loss of generality that
  $P_j$'s, hence $Q_j$'s, do not vanish on $S^1$. Also, being outer, $Q_j$'s have no zeros inside the the unit disk. Therefore, for each $j$, $\log \mid Q_j\mid$ is the real part of the holomorphic
  function $ \log Q_j$ on an open set containing the closed unit disk. By the mean value property of harmonic function we see that
  $$\log \mid Q_j(0) \mid = \int_{S^1}\log \mid Q_j(z)\mid dz \rightarrow 0~~{\rm{as}} ~~j \rightarrow \infty, $$ since $\mid P_j(z)\mid \rightarrow 1$ uniformly as $j\rightarrow \infty$.
  Also, by construction $Q_j(0)$ is positive, we see that $Q_j(0) = \mid Q_j(0)\mid \rightarrow 1$ as
  $j\rightarrow \infty$. Clearly then $Q_j, j=1,2,\cdots$ is a perturbation of the sequence of constant ultraflat polynomials.\\
\end{proof}
  Despite  rather trivial nature of a sequence of  ultraflat polynomials when divided by their inner factors, their importance stems from the following questions raised by J. E. Littlewood \cite{Littlewood}.
  \begin{enumerate}[(1)]
  \item Does there exist a sequence of ultraflat polynomials $P_j, j =1,2,\cdots$ such that for each $j$,  the coefficients of $P_j$ are all equal in absolute value?
  \item Can these coefficients in addition be real ?
   \end{enumerate}
   J-P. Kahane \cite{Kahane} has answered the first question in the affirmative. The second question remains open.
 Also flat sequence of polynomials, in particular ultraflat sequence of polynomials, appear  naturally in discussion of some spectral questions in ergodic theory. The papers of  Bourgain \cite{Bourgain} and M. Guenais \cite{Guenais} are the two  early papers connecting $L^1(S^1, dz)$ flatness with spectral questions. Solution of some of these problems depends of the existence of certain kind of flat sequence of polynomials \cite{Abd-Nad} (see section 3, Remarks ).\\
 Kahane's solution can be viewed in the following way: there is an ultraflat sequence of outer polynomials $Q_j, j =1,2,\cdots$ which when multiplied by appropriate inner functions yields an ultraflat sequence of polynomials $P_j, j=1,2,\cdots$ such that for each $j$, the coefficients of $P_j$ are equal in absolute value.\\

It may seem natural to conjecture, in the light of the proposition above, that if $P_j, j=1,2,\cdots$ is a flat sequence a.e $(dz)$  then the constant terms of their outer parts
  converge  to 1. This however is false since for any given $\lambda$, $-\infty \leq \lambda \leq 0$, it is possible to give a sequence $ P_j, j=1,2,\cdots$ of  polynomials which
   is flat a.e. $(dz)$ and such that $\int_{S^1}\log \mid P_j(z)\mid dz\rightarrow \lambda $ as $j\rightarrow \infty$ ; the sequence $Q_j, j =1,2,\cdots$ of their outer parts will be flat a.e.$(dz)$ with the same property.\\

    We will derive here a necessary condition for a sequence of polynomials to be ultraflat.\\

 Consider an analytic trigonometric  polynomial $P(z)$, with $n$ terms, of $L^2(S^1,dz)$ norm 1, $P(0) > 0$. Then
   for all $z \in S^1$, $1-\epsilon \leq \mid P(z) \mid^2 \leq 1+\epsilon$ where
   $\epsilon = \sup_{z\in S^1}|| P(z)|^2 -1|$. For any
   continuous $f$ on $S^1$,
   $$(1-\epsilon)\int_{S^1}\mid f\mid^2dz \leq \int_{S^1}\mid f(z)\mid^2 \mid P(z)\mid^2dz \leq
   (1+\epsilon)\int_{S^1}\mid f(z)\mid^2dz;$$
   in particular, if  $f(z) = \sum_{j=1}^kz^{m_j}$, a sum of characters of $S^1$, then
   $$(1-\epsilon)k \leq \sum_{i=1}^k\sum_{j=1}^k \int_{S^1}z^{m_i-m_j}\mid P(z)\mid^2dz\leq (1+\epsilon)k,  \eqno (1)$$

   Now
   $$1-\epsilon \leq \mid P(z)\mid^2 = 1 + \sum_{\overset{j= -N,}{j\neq 0}}^{N}b_j z^{n_j}
   \leq 1+\epsilon,$$
   for some suitable non-zero $b_j =\overline {b_{-j}}$, and integers $n_j=-n_{-j}, -N \leq j \leq N, j\neq 0 $. Note that $N \leq n(n-1).$
   Putting $z =1$ we get

   $$  -\epsilon \leq  \sum_{\overset{j= -N,}{j\neq 0}}^N b_j \leq  \epsilon$$

   $$\Big(\sum_{\overset{j= -N,}{j\neq 0}}^Nb_j\Big)^2  \leq \epsilon^2 < \epsilon, \eqno (2)$$

   Consider now the functions $z^{n_j} - {\overline{b_j}}, -N\leq j \leq N, j \neq 0$. The gram matrix of these vectors in $L^2(S^1, \mid P(z)\mid^2dz)$ has entries
   $$\int_{S^1}z^{n_i-n_j}\mid P(z)\mid^2dz - {\overline{b_i}}b_j.$$

Sum of these entries, denoted by $r = r(P)$, can be seen to satisfy (by equations (1) and (2) above):\\
 $$2N(1-\epsilon) -\epsilon \leq r \leq 2N(1+\epsilon) + \epsilon$$
Thus the sum of the entries of the gram matrix in question is of order $N$ and diagonal entry of the $i$th row  is  $1-| b_i|^2$.\\

We record  this calculation as:

\begin{Prop}\label{prop1}
 Let $P(z)$ be an analytic polynomial of $L^2(S^1,dz)$ norm 1
and let $\mid P(z) \mid^2 = 1 + \sum_{\overset{j= -N,}{j\neq 0}}^N b_jz^{n_j}, \forall j, b_j \neq 0 $.  Let $r(P)$ denote the sum of the entries of the gram matrix of the random variable
$z^{n_j} - \overline{b_j}, -N\leq j \leq N, j \neq 0$. Then
$$2N(1-\epsilon) -\epsilon \leq r \leq 2N(1+\epsilon) + \epsilon$$
where $\epsilon = \sup_{z\in S^1}\mid\mid P(z)\mid^2 -1\mid$.
\end{Prop}
\begin{Cor}
 Let $P_n, n=1,2,\cdots $ be a sequence of polynomials of $L^2(S^1,dz)$ norm 1. Let $$\mid P_n(z) \mid^2 = 1 + \sum_{\overset{j = -N_n,}{j\neq 0}}^{N_n}b_{j,n}z^{k_{j,n}}, \forall j, b_{j,n} \neq 0.$$ Then\\ (a) if $P_n, n=1,2,\cdots$ are uniformly bounded then so are the ratios $\frac{r(P_n)}{N_n}, n=1,2,\cdots$,\\ (b) if the sequence $P_n, n=1,2,\cdots $ is ultraflat then $\frac{r(P_n)}{2N_n} \rightarrow 1$ as $n\rightarrow \infty$. In particular this holds for any ultraflat sequence of Kahane polynomials.\\
\end{Cor}

The Gauss-Fresnel polynomials and Hardy-Littlewood polynomials are defined respectively as follows
\begin{eqnarray*}
G_n(z)&=&\frac1{\sqrt{n}}\sum_{k=0}^{n-1}g(k) z^k, {\textrm{~~where~~}} g(k)=\exp\Big(\frac{\pi i k^2}{n}\Big),\\
H_N(z)&=&\frac1{\sqrt{n}}\Big(1+\sum_{k=1}^{n-1}v(k) z^k\Big), {\textrm{~~where~~}} v(k)=\exp\Big(\frac{2\pi i (ck\ln(k))}{n}\Big).
\end{eqnarray*}
Our terminology is due to the fact that the first polynomials are connected to the Gaussian sums and the Fresnel integral and the second are studied by Hardy-Littlewood in \cite{H-L}.\\

Furthermore, it is well known that the Gauss-Fresnel polynomials and Hardy-littlewood polynomials verify

$$\Big|G_N(e^{2\pi i\theta})\Big| \leq 3C\Big(\sqrt{2}+\frac1{\sqrt{2}}\Big), \forall \theta \in [0,1),  \eqno (3)$$

and

$$\Big|H_N(e^{2\pi i\theta})\Big| \leq  K, \forall \theta \in [0,1).   \eqno (4)$$

 where $C$ and $K$ are constants independent of $N$.

 These inequalities follow as an application of the van der Corput method.
A nice account on this method can be found in \cite[p.61-67]{Tichmarsh}, \cite[p.31-37]{Salem}, \cite{Graham}, \cite[p.15-18]{N-K}.  We present the proof of inequalities (3) and (4). The principal ingredient in the proof is the following lemma due to van der Corput \cite[p.199]{Zygmund}, \cite[p.15-18]{N-K}.

\begin{lem} Suppose that $f$ is a real valued function with two continuous derivatives on $[a,b]$. Suppose also that there is some $\rho>0$ such that
$$|f''(u)| \geq \rho, \forall u \in [a,b].$$ Then,
\[\Big|\sum_{a \leq n \leq b}\exp(2 \pi i f(n))\Big| \leq  \Big(|f'(b)-f'(a)|+2\Big) \Big(\frac4{\sqrt{\rho}}+3\Big).
\]
\end{lem}
It  suffice now to take in the first case $f(u)=u\theta+\ds \frac{u^2}{2N}$ with $a=0$, $b=N-1$, and in the second case
$f(u)=cu \ln(u)+u\theta$ with $[1,N]=\bigcup_{j=0}^{n-1}[2^j,2^{j+1}] \bigcup [2^n,N]$, $2^n \leq n < 2^{N+1}.$
We therefore have as a corollary of proposition 2.3

 \begin{Th}\label{th1}
  The ratios $\frac{r(G_N)}{N}, \frac{r(H_N)}{N} $ are bounded above.
\end{Th}

 Newman in \cite{Newman} established the $L^1$-flatness of the Gauss-Fresnel polynomials. Besides, Littlewood proved in \cite{Littlewood-I} that the Gauss-Fresnel polynomials converge in measure to 1. But, since the $P_N$'s are bounded, Littlewood result implies convergence of $\mid\mid P_N\mid\mid_1$ to 1 as $N\rightarrow \infty$ hence Newman's result.\\

\begin{rem}
We note that the sequence of partial sums of the power series expansion of a finite Blaschke product is ultraflat, where the equality in absolute value of all the coefficients does not hold. If the zeros of the Blaschke product are all real  then  its power series has real coefficients, so the sequence of partial sums of such a power series is an ultraflat sequence with real coefficients. We can use these ultraflat sequences to produce non-dissipative, ergodic non-singular maps with simple Lebesgue component in the spectrum of the associated unitary operator \cite{Abd-Nad1}. It is easy to show  that there is no ultraflat sequence of polynomials whose coefficients are non-negative and uniformly bounded away from 1. However, it does not seem to be known if there is a sequence of flat polynomials in almost everywhere $ (dz)$ sense with non-negative coefficients uniformly bounded away from 1. It is known that if  there is such a sequence then there exists non-dissipative, ergodic non-singular transformations  with simple Lebesgue spectrum for the associated unitary operator \cite{Abd-Nad}.
\end{rem}

Consider now the class B of polynomials of the type $\frac{1}{\sqrt N}\Big(1 + \ds \sum_{j=1}^{N-1}z^{n_j}\Big)$. As\linebreak Bourgain \cite{Bourgain} has shown the  spectrum of a measure preserving
rank one transformation is given (up to possibly  some discrete points) by a generalized Riesz product made out of such polynomials. It is not known if there is a measure preserving rank one transformation with simple Lebesgue spectrum (in orthocomplement of constant functions).
This is equivalent to the question if there exist a flat sequence of polynomials in a.e. sense from the class B \cite{Abd-Nad}. We will give below a necessary condition for a sequence of polynomials  from the class B to be flat a.e. $(dz)$  which contrasts with the necessary condition for ultraflat sequences derived in section 1.
Consider a sequence of distinct polynomials $P_j, j=1,2,\cdots$ of the type
$$P_j(z) = \frac{1}{\sqrt{m_j}}\Big(1 +\sum_{k=1}^{m_j-1}z^{R_{k,j}}\Big),  \eqno (3) $$
Such a sequence can not be ultraflat since $P_j(1) = \sqrt{m_j}\rightarrow \infty$ as $j\rightarrow \infty$. As mentioned above it is not known if such a sequence can be flat a.e. $(dz)$.
 However,  we will show in next section that if a sequence of polynomials $P_j, j=1,2,\cdots $ of this  type converges to 1 a.e.$(dz)$ then  $\frac{r_j}{N_j} \rightarrow \infty$ as $j\rightarrow \infty$. This will follow from a more general result we prove below (Theorem 4.1 ). We will  need some ideas and results about generalized product \cite{Abd-Nad} which we give below, where Lemma 4.2 is new.

\section{Dissociated Polynomials and Generalized Riesz Products}

 Consider the following two products:
\begin{eqnarray*}
(1+z)(1+z) &=& 1+z+z +z^2 = 1 +2z +z^2,\\
(1+z)(1+z^2) &=& 1+z+z^2+z^3.
\end{eqnarray*}
In the first case we group terms with the same power of $z$, while in the second case all the powers of $z$ in the formal expansion are distinct. In the second case we say that the polynomials $1 + z$ and $1 + z^2$ are dissociated. More generally we say that a set of trigonometric polynomials is dissociated if in the formal expansion of product of any finitely many of them, the powers of $z$ in the non-zero terms are all distinct \cite{Abd-Nad}.\\

If $P(z) = \ds \sum_{j=-m}^m a_jz^j, Q(z) = \ds \sum_{j=-n}^{n}b_jz^j$, $m \leq n$, are two trigonometric  polynomials then
for some $N$, $P(z)$ and $Q(z^N)$ are dissociated. Indeed
$$P(z)\cdot Q(z^N) = \sum_{i=-m}^m\sum_{j=-n}^n a_ib_jz^{i+Nj}.$$
If we choose $N > 2n$, then we will have two exponents, say $i +Nj$ and $u+Nv$, equal if and only if $i-u = N(v-j)$ and since  $N$ is bigger than $2n$, this can happen if and only if $i=u$ and $j=v$. More generally, given any sequence $P_1, P_2, \cdots$ of polynomials one can find integers $1 = N_1 < N_2 < N_3 < \cdots,$  such that $P_1(z^{N_1}), P_2(z^{N_2}), P(z^{N_3}), \cdots$ are dissociated. Note that since the map $z \longmapsto z^{N_i}$ is measure preserving, for any $p > 0$  the $L^p(S^1, dz)$ norm of $P_i(z)$ and $P_i(z^{N_i})$ remain the same, as also their logarithmic integrals, i.e, $\int_{S^1} \log \mid P_i(z) \mid dz = \int_{S^1}\mid \log \mid P_i(z^{N_i})\mid dz $.\\

Now let $P_1, P_2,\cdots$ be a sequence of polynomials, each $P_i$ being  of $L^2(S^1, dz)$ norm 1. Then the constant term of each $\mid P_i(z)\mid^2$ is 1. If we choose  $1 = N_1 < N_2 < N_3 \cdots$  so that $\mid P_1(z^{N_1})\mid^2, \mid P_2(z^{N_2})\mid^2, \mid P_3(z^{N_3})\mid^2, \cdots$ are dissociated, then the constant term of each finite product
$$\prod_{j=1}^n\mid P_i(z^{N_i})\mid^2$$ is one so that each finite product  integrates to 1 with respect to $dz$. Also, since $\mid P_j(z^{N_i}) \mid^2, j =1,2, \cdots$ are dissociated,
for any given $k$, the $k$-th Fourier coefficient of $\prod_{j=1}^n\mid P_j(z^{N_j})\mid^2$ is either zero for all $n$, or, if it is non-zero for some $n = n_0$ (say), then its remains the same  for all $n \geq n_0$. Thus the measures $(\prod_{j=1}^n| P_j(z^{N_j})|^2)dz, n=1,2,\cdots$ admit a weak limit
on $S^1$. It is called the generalized Riesz product of the polynomials $\mid P_j(z^{N_j})\mid^2, j=1,2,\cdots$. Let $\mu$ denote this measure. It is known \cite{Abd-Nad} that
$ \prod_{j=1}^k |P_j(z^{N_j})|, k=1,2,\cdots$, converge in $L^1(S^1,dz)$ to
$\sqrt{\frac{d\mu}{dz}}$ as $k\rightarrow \infty$. It follows from this that if $\prod_{j=1}^k\mid P(z^{N_j})\mid , j=1,2,\cdots$ converge a.e. $(dz)$ to a finite positive value then $\mu$ has a part which is equivalent to Lebesgue measure.

\section{Flat a.e.$(dz)$ Sequence of Polynomials: A Necessary Condition}

    Consider a polynomial of norm 1 in $L^2(S^1,dz)$. Such a polynomial with $m$ non-zero coefficients can be written as:\\
 $$P(z) = \epsilon_0\sqrt {p_0} + \epsilon_1\sqrt {p_1}z^{R_1} + \cdots + \epsilon_{m-1}\sqrt{p_{m-1}}z^{R_{m-1}},  ~~~\eqno (4)$$
 where each $p_i$ is positive and their sum is 1, and where $\epsilon_i$'s are complex numbers of absolute value 1. Such a $P$ gives a probability measure $\mid P(z)\mid^2dz$ on the circle group which we denote by $\nu$. Now $\mid P(z)\mid^2$ can be written as
$$\mid P(z)\mid^2 = 1 + \sum_{\overset{k=-N,}{k\neq 0}}^{N}a_kz^{n_k} ,$$
where each $n_k$ is of the form $R_l - R_k,$ and each $a_k$ is a sum of terms of the type ${\epsilon_i{\overline{\epsilon_j}}\sqrt{p_i}}{\sqrt{p_j}}, i\neq j$, $a_i = \overline {a_{-i}}, 1\leq i \leq N$
We will write $$L = \sum_{\overset{k=-N,}{k\neq 0}}^{N}a_k = \mid P(1)\mid^2 -1.$$\\

Consider the special case when each $\epsilon_i =1$. Then
 $$L = \sum_{\overset{0 \leq i,j\leq m-1,}{ i\neq j}}\sqrt {p_i}\sqrt{p_j},$$ is a function of
  probability vectors $(p_0, p_1, p_2, \cdots p_{m-1})$, which  attains its maximum value when each $p_i = \frac{1}{m}$, and the maximum value is $\frac{m(m-1)}{m} = {m-1}$. We also note that $m - 1 \leq N \leq m(m-1)$. So, when $p_i$'s are all equal and = $\frac{1}{m}$ we have $$\frac{N}{L^2} \leq  \frac{m}{m-1} \leq 2$$ for $m \geq 2$

  Throughout this section we will assume that $\frac{N}{L^2} \leq A$ for some constant $A$ independent of $m$, a hypothesis which  holds when non-zero coefficients of $P$ are
  positive and equal in magnitude. Note that the requirement that $\frac{N}{L^2}$ be bounded implies that $L$ can not be close to zero, which in turn implies that a sequence of such polynomials stays away from 1 in absolute value at $z=1$, and so can not be ultraflat.\\

Consider the random variables  $X(k) =z^{n_k} - \overline{a_k}$ with respect to the measure $\nu$. We write $m(k,l) = \int_{S^1}X(k){\overline{X}}_ld\nu$, $-N \leq k,l\leq N, k, l \neq 0$ and $M$ for the correlation matrix with entries $m(k,l), -N \leq k,l\leq N, k, l \neq 0$. Let $r$ denote the sum of the entries of the matrix $M$. Note that $r= (MV, V)$, where $V$ is the $1\times 2N$ vector with all entries 1, so that $M \geq 0$, since a correlation matrix is non-negative definite. \\

We will now consider a sequence $P_j(z), j=1,2,\cdots$ of polynomials of the type (4).
The symbols $m_j, p_j(k), a_{k,j}, n_{k,j}, N_j, L_j, m_j(k,l), r_j$ will then have the obvious meaning.\\

   One of the purpose of this paper is to prove the following theorem.\\

 \begin{Th}\label{th7}
 If $\frac{N_j}{L_j^2}, j = 1,2, \cdots$ remain bounded and  $~~\ds \lim_{j\rightarrow \infty}|P_j(z)| = 1$ a.e. $(dz)$ then
 $\frac{r_j}{N_j}\rightarrow \infty$ as $j\rightarrow \infty$.
\end{Th}

(Note that $\frac{N_j}{L_j^2} = \ds \frac{N_j}{(\mid P_j(1)\mid^2 -1)^2}$.)\\

We need the following lemma based on a method devised by Peyri\`ere \cite{Peyriere}\\

\begin{lem}\label{lem1}
  If $P_j(z), j =1,2,\cdots$ is a sequence of polynomials of $L^2(S^1, dz)$ norm 1 such that the squares of their absolute values are dissociated and
  $$\sum_{j=1}^\infty \min\Big\{ 1,\sqrt{\frac{N_j}{r_j}}\Big\}=\infty.$$ Then the weak limit $\mu$ of the measures $\prod_{j=1}^n| P_j(z)|^2dz, n=1,2,\cdots$, is singular to Lebesgue measure.
\end{lem}

 \begin{proof}

  Write $s_j = \min \Big\{1, \sqrt{\frac{N_j}{r_j}}\Big\}$. Since $\sum_{j=1}^\infty s_j = \infty$, by Banach-Steinhaus theorem there is an $l^2$ sequence $\lambda_j, j=1,2,\cdots$ of positive real numbers
such that $$\sum_{j=1}^\infty \lambda_js_j = \infty.  ~~~\eqno (5)$$
Consequently, since $\frac{N_j}{L_j^2} $'s are assumed to be bounded,
$$\sum_{j=1}^\infty \frac{\lambda_j^2s_j^2}{ L^2_j}N_j < \infty  ~~~~~\eqno (6)$$

 Let $A_j = \{n_{k,j}: -N_j\leq k \leq N_j, k \neq 0\}$.
Let $V_j$ be the $1\times A_j$ matrix with all entries equal to $\frac{\lambda_j s_j}{L_j}.$
The squared Euclidean norm of this vector is $\frac{\lambda_j^2s_j^2}{L_j^2}\times 2N_j$ which when summed over $j$ is convergent by equation (6) above. Let $U_j$ be the $1\times A_j$ matrix with entries $u(n_k, j) = a_{k,j}, -N_j \leq k \leq N_j, k\neq 0$. Then the dot product $U_j\cdot V_j = L_j\times \frac{\lambda_js_j}{L_j}$ which diverges when summed over $j$ by the choice of $\lambda_j$'s.\\

Let $$f_n = \sum_{j=1}^n\sum_{k\in A_j}\frac{\lambda_js_j}{ L_j}z^{n_k},$$
$$g_n = \sum_{j=1}^n\sum_{k\in A_j}\frac{\lambda_js_j}{L_j}(X_j(k)).$$
$$= \sum_{j=1}^n\sum_{k\in A_j}\frac{\lambda_k s_j}{L_j}(z^{n_k} - \overline{a_{k,j}})$$
  Now for $m  < n$
  $$ \int_{S^1}\mid f_n - f_m\mid^2dz = \sum_{j=m+1}^n\sum_{k,l \in A_j}\frac{\lambda_j^2s_j^2}{L^2_j} \int_{S^1}z^{n_k-n_l}dz = \sum_{j=m+1}^n2\frac{\lambda_j^2s_j^2}{L^2_j}N_j$$
  $$\rightarrow 0 ~~~{\rm{as}} ~~~m,n \rightarrow \infty,$$
  and  under the assumption  of dissociation of the polynomials $\mid P_j\mid^2, j =1,2,\dots$,

  $$\int_{S^1}\mid g_n - g_m\mid^2d\mu =\sum_{j=m+1}^n\sum_{k,l\in A_j}\frac{\lambda_j^2s_j^2}{L^2_j}m_j(k,l) \leq \sum_{j=m+1}^n\frac{\lambda_j^2s_j^2}{L_j^2}r_j$$
  (since $s_j = \min\Big\{1,\sqrt{\frac{N_j}{r_j}}\Big\}$, in case $s_j = 1$, we have $N_j\geq r_j$, otherwise $s_j^2 = \frac{N_j}{r_j}$,)
  $$\leq \sum_{j=m+1}^n\frac{\lambda_j^2}{L_j^2}N_j\rightarrow 0 ~~{\rm{as}} ~~~~ m,n \rightarrow \infty$$

  We conclude that $f_n$'s converge in $L^2(S^1, dz)$ to a function whose norm is $\sum_{j=1}^\infty2\frac{\lambda_j^2s_j^2}{L_j^2}\times N_j$, and $g_n$'s converge in $L^2(S^1, d\mu)$ to a function whose norm is no more than $\sum_{j=1}^\infty\frac{\lambda^2_j}{L_j^2}N_j$.
If $\mu$ is not singular to $dz$, then there is a sequence of $l_k, k=1,2, \cdots$ of natural numbers and a $z_0\in S^1$ such that $f_{l_k}(z_0), g_{l_k}(z_0), k=1,2,\cdots$ converge to a finite limits, which in turn implies that
$$f_{l_k}(z_0) - g_{l_k}(z_0) = \sum_{j=1}^{l_k}\sum_{u \in A_j}\frac{\lambda_js_j}{L_j}a_{u,j}$$
 $$= \sum_{j=1}^{l_k} \lambda_j s_j$$ is convergent as $k\rightarrow \infty$ contrary to equation (5).\\
\end{proof}
\section{a.e.($dz)$ Flat Sequences and Generalized Riesz Product}

\begin{Th}\label{th7}\cite{Abd-Nad}.   Let $P_j, j =1,2,\cdots$ be a sequence of non-constant polynomials
of $L^2(S^1,dz)$ norm 1 such that $\lim_{j\rightarrow \infty}\mid P_j(z)\mid =1 $ a.e. $(dz)$ then there exists a subsequence $P_{j_k}, k=1,2,\cdots$ and natural numbers $l_1 < l_2 < \cdots$ such that the polynomials $P_{j_k}(z^{l_k}), k=1,2,\cdots $ are dissociated and the infinite product $\prod_{k=1}^\infty | P_{j_k}(z^{l_k})|^2$ has finite nonzero value a.e $(dz)$.
\end{Th}

\begin{proof}
 Since  $\mid P_j(z)\mid \rightarrow 1$ as $j \rightarrow \infty$ a.e. $(dz)$, by Egorov's theorem we can extract a subsequence $P_{j_k}, k = 1,2,\cdots$ such that
the sets $$E_k \egdef \Big\{z: \mid (1- \mid P_{j_l}(z)\mid )\mid < \frac{1}{2^k} ~~\forall ~~l \geq k \Big\} $$
increase to $S^1$ (except for a $dz$ null set), and $\ds \sum_{k=1}^\infty (1 -dz(E_k)) < \infty.$
 Write $Q_k = P_{j_k}$. Then for $z \in E_k$, $|1 - |Q_k(z)|| < \frac{1}{2^k}$. We choose a sequence  $l_1 <  l_2, \cdots$  of natural numbers such that
the polynomials $| P_{j_k}(z^{l_k})|^2, k=1,2,\cdots$ are all dissociated.
Consider
$$\prod_{k=1}^n | Q_k(z^{l_k}) |^2$$. We show that $\lim_{K\rightarrow \infty}\prod_{k=1}^K\mid Q_k(z^{l_k})\mid$ exists and
is nonzero a.e. $(dz)$.\\

 Now the maps $S_k: z \longmapsto z^k, k=1,2,\cdots$ preserve the measure $(dz)$, and since $\sum_{k=1}^\infty dz(S^1 -E_k) < \infty$ we have
 $\sum_{k=1}^\infty dz(S^{-l_k}(S^1 - E_k)) < \infty$. Let $F_k = S^{-l_k}(S^1 - E_k)$ and
 $F = \limsup_{k\rightarrow \infty} F_k = \cap_{k=1}^\infty\cup_{l=k}^\infty F_l$.
 Then $dz(F) =0$. If $z\notin F$, $z \notin S^{-l_k}(S^1 - E_k)$ hold for all but
 finitely many $k$, which in turn implies that $S^{l_k}z \in E_k$ for all but finitely many $k$.
 Thus, if $z \notin F$, then $\mid (1 - \mid Q_k(z^{l_k})\mid)\mid < \frac{1}{2^k}$ for all but finitely many $k$. Also the set of points $z$ for which some finite product $\prod_{k=1}^M\mid Q_k(z^{l_k})\mid$ vanishes is countable. Clearly
$\lim_{M\rightarrow \infty}\prod_{k=1}^M| Q_k(z^{l_k})|$
is nonzero a.e. $(dz)$ as claimed.\\

\end{proof}

\section{Proof of Theorem 4.1.}

Let $\frac{r_j}{N_j}, j=1,2,\cdots$ be as in the theorem. Assume that this sequence  does not tend to $\infty$ as $j \rightarrow \infty$, then over a subsequence the ratios $\frac{r_j}{N_j}, j=1,2,\cdots$ remain bounded. Without loss of generality we assume that the ratios $\frac{r_j}{N_j}, j=1,2,\cdots$ are bounded. So the reciprocal ratios  $\frac{N_j}{r_j}, j=1,2,\cdots$ sum to infinity over every subsequence.
Under the hypothesis that $\mid P_j\mid \rightarrow 1$ as $j\rightarrow \infty$, by theorem 5.1, we get a subsequence
$P_{j_k} = Q_k, k=1,2, \cdots$ and natural numbers $l_1 < l_2 < \cdots$ such that the polynomials
$\mid Q_k((z^{l_k})\mid^2, k=1,2,\cdots $ are dissociated and the infinite product
$\prod_{k=1}^\infty\mid Q_k(z^{l_k})\mid^2$ has finite non-zero limit a.e. $(dz)$. Also, since  $\mid Q_k(z^{l_k})\mid^2$'s are dissociated,  the measures $\mu_n \egdef \prod_{k=1}^n\mid Q_k(z^{l_k})\mid^2dz$ converge weakly to a measure $\mu$ on $S^1$
for which $\frac{d\mu}{dz} > 0$ a.e $(dz)$, indeed
$$\frac{d\mu}{dz} = \prod_{k=1}^\infty\mid Q(z^{l_k})\mid^2 ~~~a.e. (dz) .$$
 Since the map $z\rightarrow z^{l_k}$ preserves the Lebesgue measure on $S^1$,  the $m_{j_k}(u,v)$'s for  $\mid P_{j_k}(z)\mid^2dz$ are the same as for  $\mid P_{j_k}(z^{l_k})\mid^2dz$. Since $\sum_{k=1}^\infty \sqrt{\frac{N_{j_{k}}}{r_{j_k}}} = \infty$, by Lemma 4.2 $\mu$ is singular to $(dz)$. The contradiction implies that $\frac{r_j}{N_j} \rightarrow \infty$ as $j \rightarrow \infty$.
 This completes the proof of the theorem.

\begin{Cor}{\label{cor1}}
Let $P_j = \ds \frac{1}{\sqrt{m_j}}\Big(1 + \sum_{i=1}^{m_j-1}z^{R_{i,j}}\Big), j =1,2,\cdots$. If the infinite product $\prod_{j=1}^\infty\mid P_j(z)\mid^2$ has finite non-zero value a.e. $(dz)$ then $\frac{r_j}{N_j}\rightarrow \infty$ as $j \rightarrow \infty$.   If the sequence $\frac{r_j}{N_j}, j =1,2,\cdots$ is  bounded  then the infinite product does not have a finite non-zero limit on a set of positive $dz$ measure.
\end{Cor}

\begin{Cor}{\label{ cor1}}
 If the quantities $\frac{N}{r}$ for the polynomials of the type
  $$P(z) = \frac{1}{\sqrt m}\Big(1 + z^{R_1} + z^{R_2} + \cdots + z^{R_{m-1}}\Big)$$
is bounded away from zero then a generalized product associated with a measure preserving
 rank one map can not have a Lebesgue component.

\end{Cor}

\begin{proof}

For if there is a Lebesgue component for some such generalized Riesz product,then there is a sequence of polynomials of the given type which converge in absolute value to 1 a.e. $(dz)$ so that the quantities $\frac{r}{N}$ for such polynomials diverge, contradicting the
hypothesis.
\end{proof}

\end{document}